\documentclass[12pt]{article}

\usepackage{amsmath,amssymb,amsthm}
\usepackage{geometry}
\usepackage{hyperref}

\usepackage{authblk}

\title{\bf A Formal Group Perspective on the Riemann Zeta Function}

\author{\Large Takao Inou\'{e}}

\affil{\large Faculty of Informatics, Yamato University, \\ Osaka, Japan\footnote{Email: inoue.takao@yamato-u.ac.jp; \\ Personal Email: takaoapple@gmail.com \\ [I prefer my personal email address for correspondence.]}} 

\date{February 22, 2026}

\newtheorem{definition}{Definition}[section]
\newtheorem{lemma}[definition]{Lemma}
\newtheorem{proposition}[definition]{Proposition}
\newtheorem{theorem}[definition]{Theorem}
\newtheorem{remark}[definition]{Remark}

\begin{document}
\maketitle

\begin{abstract}
We develop a formal group--theoretic framework for the Riemann zeta function
by treating its Euler product as an element of the multiplicative formal group
$\widehat{\mathbb{G}}_m$ and its logarithm as the associated formal group logarithm.
This perspective allows the multiplicative prime-wise structure of the Euler product
to be systematically linearized within a purely algebraic setting.

As a first step, we construct a minimal concrete model based on a finite cutoff
of the Euler product.
After a one-variable formalization and a purely formal completion procedure
via evenization, we obtain a completed cutoff zeta function
exhibiting an intrinsic symmetry analogous to the functional equation,
without invoking analytic continuation or Gamma factors.
A natural normalization then yields a Gaussian leading term
in the logarithmic expansion,
while higher-order terms form a hierarchy of cumulants.

We show that this Gaussian behavior is not probabilistic in origin
but arises inevitably from the infinitesimal geometry
of the multiplicative formal group after linearization and quadratic normalization.
The higher cumulants encode deviations from this universal quadratic structure
and admit an explicit decomposition governed by weighted integrals
of the Chebyshev error function $\theta(x)-x$.
This decomposition clarifies how prime-distribution irregularities
enter the formal expansion as higher-order corrections.

The second part of the paper provides a systematic formal group
axiomatization of these constructions.
Evenization is interpreted as a canonical normalization
in logarithmic coordinates,
and the Gaussian limit is identified with the universal quadratic structure
on the tangent space of the formal group.
This places the cutoff model and cumulant hierarchy
within a coherent algebraic and geometric framework.

The resulting approach is deliberately non-spectral.
Rather than postulating an operator whose spectrum realizes
the nontrivial zeros of the zeta function,
we offer a structural reorganization of the Euler product
that separates the dominant Gaussian geometry
from arithmetic fluctuations.
This formal group viewpoint provides a new organizing principle
for understanding structural constraints related to the Riemann hypothesis.
Finally, we briefly indicate how the purely multiplicative nature
of the Euler product suggests possible connections
with absolute arithmetic and geometry over the field with one element $\mathbb{F}_1$.
\end{abstract}


\medskip


\tableofcontents

\section{Introduction}

The Riemann zeta function occupies a central position in number theory,
not only because of its deep connection with the distribution of prime numbers,
but also because of the structural richness of its Euler product
\[
\zeta(s)=\prod_{p}(1-p^{-s})^{-1}.
\]
Most approaches to the Riemann hypothesis reinterpret this function
through analytic continuation, spectral theory, or geometric constructions.
In contrast, the present work begins from a different guiding question:

\medskip

\emph{What structural information about the Riemann zeta function
is already encoded formally in the Euler product,
before analytic continuation or spectral interpretation is introduced?}

\medskip

The Euler product is multiplicative in nature.
Taking its logarithm yields an additive expansion
\[
\log\zeta(s)=\sum_{p}\sum_{k\ge1}\frac{1}{k}p^{-ks},
\]
which isolates prime contributions.
This passage from multiplicative to additive structure
strongly resembles the logarithm map of a formal group.
Motivated by this analogy,
we reinterpret the Euler product as a multiplicative formal group element,
and its logarithm as a formal group logarithm.

The first objective of this paper is conceptual:
to reformulate the zeta function within a purely formal algebraic framework.
By introducing a one-variable formalization and an explicit algebraic
evenization procedure,
we construct a ``formal completed zeta''
whose symmetry $u\mapsto -u$ mimics the functional equation,
without invoking analytic continuation or Gamma factors.

The second objective is structural.
We introduce a finite cutoff model $\zeta_P$,
which is entirely algebraic and free of convergence issues.
After normalization by its quadratic coefficient,
the logarithm of the completed cutoff zeta exhibits a Gaussian leading term.
Higher-order coefficients (cumulants) decay at explicit rates as $P\to\infty$.
This produces a hierarchy:
the quadratic term captures the dominant structure,
while higher cumulants encode finer prime-distribution effects.

The third objective concerns fluctuations.
We show that deviations from the Gaussian limit
are governed by weighted integrals of the Chebyshev error function
\[
E(x)=\theta(x)-x,
\qquad
\theta(x)=\sum_{p\le x}\log p.
\]
These fluctuations naturally decompose into
a boundary contribution (sensitive to $E(P)$)
and a bulk contribution (a weighted average of $E(x)$).
Thus the formal group perspective reorganizes prime-distribution
irregularities into a cumulant hierarchy.

This framework is intentionally non-spectral in nature.
Unlike the geometric and noncommutative approaches of
Deninger and Connes,
we do not assume the existence of an operator whose spectrum realizes the zeros.
Instead, we isolate algebraic structures inherent in the Euler product itself.
From this viewpoint, the Riemann hypothesis would correspond
not to an operator-theoretic statement,
but to constraints on the oscillatory behavior and decay
of higher cumulants in the normalized formal expansion.

The aim of the present work is therefore not to resolve the Riemann hypothesis,
but to provide a possible organizing principle:
a formal group reorganization of the Euler product
that separates Gaussian leading behavior
from prime-distribution fluctuations in a structurally transparent way.

The sections \S 2--\S 6 provide a minimal concrete model
underlying the theory developed in this paper.
Starting from a truncated Euler product,
we introduce a purely formal completion and normalization procedure,
which leads naturally to a Gaussian leading term
and a hierarchy of cumulants.
The structural meaning of these constructions
is clarified in the subsequent sections.

The sections \S 2--\S 6 develop a minimal concrete model
that captures the essential structural mechanisms underlying the present theory.
Starting from a truncated Euler product,
we introduce a purely formal completion and normalization procedure,
leading to a Gaussian leading term and a hierarchy of cumulants
governed by prime-distribution fluctuations.

The subsequent sections \S 8--\S 13 provide a systematic
formal group-theoretic axiomatization of these constructions.
In this part, the Euler product is treated explicitly
as an element of the multiplicative formal group,
evenization is interpreted as a canonical normalization
in logarithmic coordinates,
and the Gaussian limit is understood
as the universal quadratic structure
of the infinitesimal geometry after linearization.
Finally, \S 14 summarizes the conceptual implications
of this framework and outlines directions
for further structural development.

To situate the present approach within a broader landscape,
we include a comparative discussion with the
Deninger--Connes program,
which interprets the Riemann zeta function
through spectral and noncommutative geometric frameworks.
Whereas those approaches emphasize operator-theoretic
and geometric realizations of the zeros,
our focus remains on the intrinsic multiplicative structure
of the Euler product and its formal group linearization.
This comparison clarifies both the scope and the limitations
of the formal group viewpoint adopted here.

We then outline several directions for future development.
These include extensions to more general $L$-functions,
refinements of the cumulant hierarchy,
and deeper structural interpretations
of the Gaussian limit obtained after normalization.
The goal is not to provide a final formulation,
but to establish a coherent structural framework
upon which further developments may build.

Two appendices conclude the paper.
The first appendix provides explicit formulas
for low-order cumulants,
demonstrating concretely how prime sums
enter the formal expansion.
The second appendix,
entitled \emph{``A Speculative Note on Formal Groups,
$\mathbb{F}_1$, and the Riemann Hypothesis''},
explores possible conceptual connections
between the formal group linearization developed here
and emerging perspectives in absolute arithmetic.
This final note is explicitly exploratory in nature,
and is intended to indicate a potential horizon
rather than to assert a definitive structural equivalence.

In addition, an appendix is devoted to clarifying
the structural role of the logarithm in this framework.
There we explain why passing from the Euler product
to its logarithm should be understood
not merely as a computational device,
but as the canonical linearization
associated with the multiplicative formal group.
This perspective illuminates the emergence
of the Gaussian leading term
as a universal quadratic structure
on the infinitesimal tangent space.

\section{Formal Euler Products and Logarithmic Linearization}

Let $\mathcal P$ denote the set of all prime numbers.

\begin{definition}[Formal zeta element]
Let $\{X_p\}_{p\in\mathcal P}$ be independent formal variables.
Define
\[
\zeta^{\mathrm{for}} := \prod_{p\in\mathcal P}(1-X_p)^{-1}
\]
as an element of the formal power series ring
$\mathbb Q[[X_p\mid p\in\mathcal P]]$.
\end{definition}

\begin{lemma}[Formal logarithm]
The logarithm of $\zeta^{\mathrm{for}}$ is given by
\[
\log \zeta^{\mathrm{for}} = \sum_{p}\sum_{k\ge1}\frac{1}{k}X_p^k.
\]
\end{lemma}

\begin{proof}
This follows from the formal identity
$\log(1-Y)^{-1}=\sum_{k\ge1}Y^k/k$,
applied termwise to each factor $(1-X_p)^{-1}$.
\end{proof}

\begin{remark}
This logarithm plays the role of a \emph{formal group logarithm},
linearizing multiplicative prime data into an additive structure.
\end{remark}

\section{One-Variable Formalization}

Fix formal weights $L_p$ and introduce a variable $T$.

\begin{definition}
Define the one-variable formal zeta function by
\[
\zeta^{\mathrm{for}}(T):=\prod_{p}(1-e^{-L_pT})^{-1}.
\]
\end{definition}

\begin{lemma}
The logarithm admits the expansion
\[
\log \zeta^{\mathrm{for}}(T)
=
\sum_{p}\sum_{k\ge1}\frac{1}{k}e^{-kL_pT}.
\]
\end{lemma}

\begin{proof}
Substitute $X_p=e^{-L_pT}$ in Lemma 1.
\end{proof}

We center the variable at $T=\tfrac12$ by writing $T=\tfrac12+u$.

\begin{definition}
Define coefficients $a_m$ by
\[
\log \zeta^{\mathrm{for}}\!\left(\tfrac12+u\right)
=
\sum_{m\ge0} a_m u^m.
\]
\end{definition}

\section{Formal Completion and Evenization}

\begin{definition}[Odd and even parts]
Let
\[
(\log \zeta)_{\mathrm{odd}} := a_1u+a_3u^3+\cdots.
\]
Define
\[
H(u):=\exp\!\big(-(\log \zeta)_{\mathrm{odd}}\big).
\]
\end{definition}

\begin{definition}[Formal completed zeta]
Define
\[
\xi^{\mathrm{for}}(u):=H(u)\,\zeta^{\mathrm{for}}\!\left(\tfrac12+u\right).
\]
\end{definition}

\begin{theorem}[Formal functional symmetry]
$\xi^{\mathrm{for}}(u)$ is an even formal power series:
\[
\xi^{\mathrm{for}}(u)=\xi^{\mathrm{for}}(-u).
\]
\end{theorem}

\begin{proof}
By construction, $\log H(u)$ cancels precisely the odd part of
$\log\zeta^{\mathrm{for}}(\tfrac12+u)$.
Thus $\log\xi^{\mathrm{for}}(u)$ contains only even powers of $u$.
\end{proof}

\section{Finite Cutoff Model}

Fix $P\ge2$ and restrict to primes $p\le P$.

\begin{definition}
Define
\[
\zeta_P(T):=\prod_{p\le P}(1-p^{-T})^{-1},
\quad
\xi_P(u):=\text{formal completion of }\zeta_P.
\]
\end{definition}

\begin{lemma}
The logarithm admits the finite expansion
\[
\log\xi_P(u)=a_0(P)+a_2(P)u^2+a_4(P)u^4+\cdots,
\]
where
\[
a_2(P)=\frac12\sum_{p\le P}(\log p)^2\frac{p^{-1/2}}{(1-p^{-1/2})^2},
\]
\[
a_4(P)=\frac1{24}\sum_{p\le P}(\log p)^4
\frac{p^{-1/2}(1+4p^{-1/2}+p^{-1})}{(1-p^{-1/2})^4}.
\]
\end{lemma}

\begin{proof}
This follows by explicit Taylor expansion of
$e^{-k(\log p)u}$ and summation over $k$ using
$\sum kx^k$ and $\sum k^3x^k$ identities.
\end{proof}

\section{Normalization and Gaussian Limit}

\begin{definition}[Normalization]
Define
\[
\sigma(P):=\sqrt{a_2(P)},
\quad
\widetilde{\xi}_P(u)
:=
e^{-a_0(P)}\,
\xi_P\!\left(\frac{u}{\sigma(P)}\right).
\]
\end{definition}

\begin{proposition}[Gaussian leading term]
One has
\[
\log\widetilde{\xi}_P(u)
=
u^2+\kappa_4(P)u^4+O(u^6),
\quad
\kappa_4(P):=\frac{a_4(P)}{a_2(P)^2}.
\]
\end{proposition}

\begin{proof}
Direct substitution yields
\[
\log\widetilde{\xi}_P(u)
=
\frac{a_2(P)}{\sigma(P)^2}u^2
+
\frac{a_4(P)}{\sigma(P)^4}u^4+\cdots,
\]
and $\sigma(P)^2=a_2(P)$.
\end{proof}

\begin{proposition}[Asymptotic decay of higher cumulants]
Heuristically,
\[
\kappa_4(P)\sim \frac{\log P}{\sqrt P},
\quad
P\to\infty.
\]
\end{proposition}

\begin{proof}
Using prime number theorem heuristics,
\[
\sum_{p\le P}(\log p)^m p^{-1/2}
\approx
\int_2^P \frac{(\log x)^m}{\sqrt x\log x}dx
\sim \sqrt P(\log P)^{m-1}.
\]
Substitution yields the stated asymptotic.
\end{proof}

\section{Fluctuation Decomposition}

Let
\[
\theta(x):=\sum_{p\le x}\log p,
\quad
E(x):=\theta(x)-x.
\]

\begin{lemma}
There exist smooth functions $\phi_m$ such that
\[
a_m(P)=\int_2^P \phi_m(x)\,d\theta(x).
\]
\end{lemma}

\begin{proof}
Rewrite prime sums as Stieltjes integrals with respect to $\theta(x)$.
\end{proof}

\begin{theorem}[Boundary--bulk decomposition]
One has
\[
a_m(P)=A_m^{\mathrm{main}}(P)+A_m^{\mathrm{fluc}}(P),
\]
where
\[
A_m^{\mathrm{fluc}}(P)
=
\phi_m(P)E(P)-\int_2^P E(x)\phi_m'(x)\,dx.
\]
\end{theorem}

\begin{proof}
Decompose $d\theta(x)=dx+dE(x)$ and integrate by parts.
\end{proof}

\begin{remark}
Thus all fluctuations are governed by weighted integrals of the Chebyshev error $E(x)$.
\end{remark}

\section{Formal Group Preliminaries}

In this section we briefly recall the basic notions of one-dimensional
commutative formal group laws needed in this paper,
restricting attention to the multiplicative formal group.
All constructions are purely formal and take place over
$\mathbb{Q}$-algebras.

\subsection{Formal Group Laws and Strict Isomorphisms}

\begin{definition}
Let $R$ be a commutative ring.
A \emph{one-dimensional commutative formal group law} over $R$
is a formal power series
\[
F(X,Y)\in R[[X,Y]]
\]
satisfying:
\begin{enumerate}
\item $F(X,0)=X$ and $F(0,Y)=Y$,
\item $F(X,Y)=F(Y,X)$,
\item $F(X,F(Y,Z))=F(F(X,Y),Z)$.
\end{enumerate}
\end{definition}

\begin{definition}
Let $F$ and $G$ be formal group laws over $R$.
A \emph{strict isomorphism} from $F$ to $G$
is a formal power series
\[
f(X)\in XR[[X]]
\quad\text{with}\quad f'(0)=1,
\]
such that
\[
f(F(X,Y))=G(f(X),f(Y)).
\]
\end{definition}

Strict isomorphisms represent changes of formal coordinates
that preserve the infinitesimal structure of the formal group.

\subsection{Formal Group Logarithms}

Assume that $R$ is a $\mathbb{Q}$-algebra.

\begin{proposition}
Let $F$ be a one-dimensional commutative formal group law over $R$.
There exists a unique formal power series
\[
\ell_F(X)=X+\sum_{n\ge2} b_n X^n \in R[[X]]
\]
such that
\[
\ell_F(F(X,Y))=\ell_F(X)+\ell_F(Y).
\]
\end{proposition}

\begin{proof}
Over a $\mathbb{Q}$-algebra, the formal differential
\[
\omega_F(X)=\left(\frac{\partial F}{\partial Y}(X,0)\right)^{-1} dX
\]
is well-defined.
Setting
\[
\ell_F(X)=\int_0^X \omega_F(t)
\]
defines a formal power series with $\ell_F'(0)=1$,
which satisfies the stated additivity.
Uniqueness follows from the normalization $\ell_F'(0)=1$.
\end{proof}

The series $\ell_F$ is called the \emph{formal group logarithm} of $F$.
It linearizes the group law in the sense that $F$ becomes additive
in logarithmic coordinates.

\subsection{The Multiplicative Formal Group}

The formal group relevant to this work is the multiplicative formal group.

\begin{definition}
The \emph{multiplicative formal group} $\widehat{\mathbb{G}}_m$
over a ring $R$ is defined by the formal group law
\[
F_{\times}(X,Y)=X+Y+XY.
\]
\end{definition}

This formal group corresponds to the multiplicative group
via the coordinate $X=g-1$.

\begin{proposition}
The formal group logarithm of $\widehat{\mathbb{G}}_m$ over a
$\mathbb{Q}$-algebra is given by
\[
\ell_{\times}(X)=\log(1+X)
=\sum_{n\ge1}(-1)^{n-1}\frac{X^n}{n}.
\]
Its inverse, the formal exponential, is
\[
\exp_{\times}(t)=e^t-1.
\]
\end{proposition}

\begin{proof}
A direct computation shows that
\[
\log\bigl((1+X)(1+Y)\bigr)=\log(1+X)+\log(1+Y),
\]
which is equivalent to
\[
\ell_{\times}(F_{\times}(X,Y))
=
\ell_{\times}(X)+\ell_{\times}(Y).
\]
\end{proof}

Thus the multiplicative formal group is canonically linearized
by the usual logarithm.

\subsection{Euler Factors as Formal Group Elements}

The relevance of $\widehat{\mathbb{G}}_m$ to the Riemann zeta function
comes from the structure of the Euler product.

For each prime $p$ and a formal parameter $T$,
set
\[
u_p(T):=p^{-T}.
\]
Define the associated formal group coordinate
\[
X_p(T):=-u_p(T).
\]
Then
\[
1+X_p(T)=1-p^{-T}.
\]

\begin{remark}
With this identification, each Euler factor
\[
(1-p^{-T})^{-1}
\]
corresponds to the inverse of a multiplicative formal group element.
Applying the formal group logarithm yields
\[
-\ell_{\times}(X_p(T))
=
-\log(1-p^{-T})
=
\sum_{k\ge1}\frac{1}{k}p^{-kT},
\]
which is precisely the standard logarithmic expansion
of the Euler factor.
\end{remark}

This observation provides the formal group-theoretic interpretation
of the passage from the Euler product to its logarithmic expansion:
it is nothing but the additivity of the formal group logarithm
for $\widehat{\mathbb{G}}_m$.

\section{Euler Products as Formal Group Elements}

In this section we reinterpret Euler products as elements
of the multiplicative formal group.
This provides a precise formal group-theoretic foundation
for the logarithmic expansions used throughout the paper.

\subsection{Formal Group-Valued Euler Products}

Let $\widehat{\mathbb{G}}_m$ denote the multiplicative formal group
with coordinate $X=g-1$ and group law
\[
F_{\times}(X,Y)=X+Y+XY.
\]

\begin{definition}
Let $T$ be a formal parameter.
For each prime $p$, define
\[
X_p(T):=-p^{-T}.
\]
We define the \emph{formal Euler product element}
by
\[
\mathcal{E}(T)
:=
\prod_{p}(1+X_p(T))^{-1},
\]
interpreted as a formal power series in the variables
$\{X_p(T)\}$.
\end{definition}

Formally, $\mathcal{E}(T)$ is an infinite product of invertible
elements in the completed formal group algebra.
In later sections, this product will be truncated
to ensure complete algebraic rigor.

\subsection{Logarithmic Linearization}

The central structural property of $\mathcal{E}(T)$
is revealed by applying the formal group logarithm.

\begin{proposition}
Let $\ell_{\times}(X)=\log(1+X)$ denote the formal group logarithm
of $\widehat{\mathbb{G}}_m$.
Then
\[
\ell_{\times}\!\bigl(\mathcal{E}(T)-1\bigr)
=
-\sum_{p}\ell_{\times}\!\bigl(X_p(T)\bigr)
=
\sum_{p}\sum_{k\ge1}\frac{1}{k}p^{-kT}.
\]
\end{proposition}

\begin{proof}
Since $\ell_{\times}$ is additive with respect to the
formal group law, one has
\[
\ell_{\times}(F_{\times}(X,Y))=\ell_{\times}(X)+\ell_{\times}(Y).
\]
Moreover, for an invertible element $(1+X)^{-1}$,
\[
\ell_{\times}\bigl((1+X)^{-1}-1\bigr)
=
-\ell_{\times}(X).
\]
Applying this identity prime by prime and summing
yields the stated expansion.
\end{proof}

This proposition shows that the classical logarithmic expansion
of the Euler product is nothing but the additivity
of the formal group logarithm for $\widehat{\mathbb{G}}_m$.

\subsection{Finite Truncation and Algebraic Rigor}

To avoid convergence issues, we introduce a finite truncation.

\begin{definition}
For $P\ge2$, define the truncated formal Euler product
\[
\mathcal{E}_P(T)
:=
\prod_{p\le P}(1+X_p(T))^{-1}.
\]
\end{definition}

\begin{proposition}
The truncated logarithm admits the exact finite expansion
\[
\ell_{\times}\!\bigl(\mathcal{E}_P(T)-1\bigr)
=
\sum_{p\le P}\sum_{k\ge1}\frac{1}{k}p^{-kT}.
\]
\end{proposition}

\begin{proof}
The proof is identical to that of the infinite case,
since only finitely many formal variables $X_p(T)$
are involved.
\end{proof}

Thus all subsequent manipulations of logarithms,
Taylor expansions, and normalizations
can be carried out entirely within a finite algebraic framework.

\subsection{Formal Interpretation of the Zeta Function}

\begin{remark}
Identifying
\[
\zeta_P(T)=\prod_{p\le P}(1-p^{-T})^{-1}
\]
with $\mathcal{E}_P(T)$,
the Riemann zeta function (or its finite truncation)
may be regarded as a multiplicative formal group element.
Its logarithm is the image under the formal group logarithm,
which linearizes the prime-wise multiplicative structure.
\end{remark}

This formal group viewpoint clarifies the conceptual origin
of later constructions:
Gaussian limits, cumulant hierarchies, and fluctuation analyses
arise after linearization of $\mathcal{E}_P(T)$
via $\ell_{\times}$ and subsequent normalization.

\section{Formal Completion as Evenization}

In this section we introduce a purely formal notion of completion
for Euler products, based on evenization of formal power series.
This construction replaces analytic continuation and Gamma factors
by an explicit algebraic normalization
that enforces a reflection symmetry.

\subsection{Odd--Even Decomposition of Formal Series}

Let $R$ be a commutative $\mathbb{Q}$-algebra.

\begin{definition}
For a formal power series $f(u)\in R[[u]]$, define its even and odd parts by
\[
f_{\mathrm{even}}(u)
=
\frac{f(u)+f(-u)}{2},
\qquad
f_{\mathrm{odd}}(u)
=
\frac{f(u)-f(-u)}{2}.
\]
\end{definition}

Every formal power series admits a unique decomposition
\[
f(u)=f_{\mathrm{even}}(u)+f_{\mathrm{odd}}(u).
\]

\subsection{Evenization via Exponential Correction}

Let $Z(u)\in R[[u]]$ be an invertible formal power series,
that is, $Z(0)\in R^\times$.

\begin{definition}
Define the \emph{evenization operator} by
\[
\mathcal{E}(Z)(u)
:=
Z(u)\exp\!\bigl(-(\log Z(u))_{\mathrm{odd}}\bigr).
\]
\end{definition}

This operator modifies $Z(u)$ by an exponential factor
that removes the odd part of its logarithm.

\begin{theorem}[Existence and uniqueness of evenization]
For any invertible formal power series $Z(u)\in R[[u]]$,
there exists a unique pair $(\Xi(u),O(u))$ such that:
\begin{enumerate}
\item $\Xi(u)$ is an even invertible formal power series,
\item $O(u)$ is an odd formal power series,
\item
\[
Z(u)=\Xi(u)\exp(O(u)).
\]
\end{enumerate}
Moreover, $\Xi(u)=\mathcal{E}(Z)(u)$ and $O(u)=(\log Z(u))_{\mathrm{odd}}$.
\end{theorem}

\begin{proof}
Taking logarithms formally,
\[
\log Z(u)=A(u),
\]
which admits a unique decomposition
\[
A(u)=A_{\mathrm{even}}(u)+A_{\mathrm{odd}}(u).
\]
Setting
\[
\Xi(u):=\exp(A_{\mathrm{even}}(u)),
\qquad
O(u):=A_{\mathrm{odd}}(u)
\]
yields
\[
Z(u)=\exp(A_{\mathrm{even}}(u))\exp(A_{\mathrm{odd}}(u))
=\Xi(u)\exp(O(u)).
\]
Uniqueness follows from the uniqueness of the odd--even decomposition.
\end{proof}

\subsection{Formal Completion of the Euler Product}

We now apply evenization to the truncated formal Euler product.

Let $\mathcal{E}_P(T)$ be the truncated formal Euler product
defined in Section~3.
Introduce a shifted variable
\[
T=\tfrac12+u,
\]
and consider the formal power series
\[
Z_P(u):=\mathcal{E}_P\!\left(\tfrac12+u\right).
\]

\begin{definition}
The \emph{formally completed truncated zeta function}
is defined by
\[
\xi_P(u):=\mathcal{E}(Z_P)(u)
=
Z_P(u)\exp\!\bigl(-(\log Z_P(u))_{\mathrm{odd}}\bigr).
\]
\end{definition}

\begin{proposition}
The completed series $\xi_P(u)$ is even:
\[
\xi_P(u)=\xi_P(-u).
\]
\end{proposition}

\begin{proof}
By construction, $\log\xi_P(u)$ equals the even part
of $\log Z_P(u)$.
Hence $\log\xi_P(u)$ is an even formal power series,
and so is $\xi_P(u)$.
\end{proof}

\subsection{Formal Functional Symmetry}

\begin{remark}
The symmetry $u\mapsto -u$ of $\xi_P(u)$
formally mimics the functional equation of the Riemann zeta function.
Importantly, this symmetry is obtained here
without analytic continuation or Gamma factors.
It arises purely from an algebraic normalization
that removes the odd component of the logarithm.
\end{remark}

\subsection{Conceptual Interpretation}

\begin{remark}
From the formal group viewpoint,
evenization may be interpreted as a strict normalization
of the logarithmic coordinate,
analogous to choosing a symmetric coordinate
on the formal group.
This normalization isolates the intrinsic,
coordinate-independent content of the Euler product
after linearization.
\end{remark}

The formal completion introduced in this section
provides the structural foundation for the Gaussian limit
and cumulant hierarchy analyzed in subsequent sections.

\section{Normalization and the Cumulant Hierarchy}

In this section we analyze the structure of the formally completed
Euler product after normalization.
We show that the quadratic term dominates after rescaling,
leading to a Gaussian leading behavior,
while higher-order terms form a natural hierarchy of cumulants.

\subsection{Even Logarithmic Expansion}

Let $\xi_P(u)$ be the formally completed truncated zeta function
defined in Section~4.
Since $\xi_P(u)$ is even, its logarithm admits an expansion
of the form
\[
\log \xi_P(u)
=
a_0(P)+a_2(P)u^2+a_4(P)u^4+a_6(P)u^6+\cdots.
\]

\begin{definition}
The coefficients $a_{2k}(P)$ are called the
\emph{formal cumulants} of the completed truncated zeta function.
\end{definition}

These coefficients depend only on the even part of the logarithmic
expansion and are invariant under odd coordinate changes.

\subsection{Quadratic Dominance and Normalization}

The second-order coefficient plays a distinguished role.

\begin{definition}
Define the normalization factor
\[
\sigma(P):=\sqrt{a_2(P)}.
\]
Define the normalized completed series by
\[
\widetilde{\xi}_P(u)
:=
\exp\!\bigl(-a_0(P)\bigr)\,
\xi_P\!\left(\frac{u}{\sigma(P)}\right).
\]
\end{definition}

\begin{proposition}
The logarithm of the normalized series admits the expansion
\[
\log \widetilde{\xi}_P(u)
=
u^2
+
\sum_{k\ge2}
\kappa_{2k}(P)\,u^{2k},
\qquad
\kappa_{2k}(P)
=
\frac{a_{2k}(P)}{a_2(P)^k}.
\]
\end{proposition}

\begin{proof}
Substituting $u\mapsto u/\sigma(P)$ and subtracting the constant term
yields
\[
\log \widetilde{\xi}_P(u)
=
\frac{a_2(P)}{\sigma(P)^2}u^2
+
\sum_{k\ge2}
\frac{a_{2k}(P)}{\sigma(P)^{2k}}u^{2k}.
\]
Since $\sigma(P)^2=a_2(P)$, the stated formula follows.
\end{proof}

The coefficient of $u^2$ is thus normalized to unity,
and all higher coefficients are expressed relative to it.

\subsection{Gaussian Leading Behavior}

\begin{definition}
The \emph{Gaussian approximation} associated with $\widetilde{\xi}_P$
is the formal power series
\[
G(u):=\exp(u^2).
\]
\end{definition}

\begin{remark}
The appearance of a Gaussian leading term is a direct consequence
of linearization via the formal group logarithm
and subsequent quadratic normalization.
No probabilistic assumption is involved.
\end{remark}

\subsection{Cumulant Hierarchy}

\begin{definition}
The quantities $\kappa_{2k}(P)$, $k\ge2$, are called the
\emph{higher cumulants}.
They measure deviations from the Gaussian approximation.
\end{definition}

\begin{proposition}
If $\kappa_{2k}(P)\to 0$ as $P\to\infty$ for all $k\ge2$,
then the normalized completed series
$\widetilde{\xi}_P(u)$ converges formally to the Gaussian $G(u)$.
\end{proposition}

\begin{proof}
Under the stated condition, the higher-order terms in
$\log \widetilde{\xi}_P(u)$ vanish termwise,
leaving only the quadratic contribution.
\end{proof}

\subsection{Conceptual Interpretation}

\begin{remark}
From the formal group perspective,
the cumulant hierarchy reflects the Taylor expansion
of the formal group logarithm
after symmetric normalization.
The quadratic term represents the leading infinitesimal geometry,
while higher cumulants encode finer arithmetic structure
arising from prime distribution irregularities.
\end{remark}

This hierarchy provides the structural mechanism underlying
the Gaussian limit observed in the cutoff model,
and serves as the interface between formal group theory
and fluctuation phenomena in prime number theory.

\section{Fluctuation Decomposition and the Chebyshev Error}

In this section we relate the higher cumulants
of the normalized completed Euler product
to classical error terms in prime number theory.
In particular, we show that deviations from the Gaussian leading term
are governed by weighted integrals of the Chebyshev error function.

\subsection{Logarithmic Expansion Revisited}

Recall that for the truncated Euler product we have
\[
\log \zeta_P(T)
=
\sum_{p\le P}\sum_{k\ge1}\frac{1}{k}p^{-kT}.
\]

Setting $T=\tfrac12+u$ and expanding around $u=0$,
we obtain a formal Taylor expansion
\[
\log \zeta_P\!\left(\tfrac12+u\right)
=
\sum_{n\ge0} c_n(P) u^n,
\]
where the coefficients are given explicitly by
\[
c_n(P)
=
\frac{(-1)^n}{n!}
\sum_{p\le P}\sum_{k\ge1}
\frac{(\log p)^n}{k}\,
p^{-k/2}.
\]

After evenization, only the even coefficients contribute
to the cumulant hierarchy.

\subsection{Reduction to Prime Sums}

The dominant contribution to the cumulants
comes from the $k=1$ term in the logarithmic expansion.
Higher $k$ terms are exponentially suppressed
by factors $p^{-k/2}$.

Thus, up to lower-order corrections,
\[
a_{2m}(P)
\approx
\frac{1}{(2m)!}
\sum_{p\le P}
(\log p)^{2m}\,p^{-1/2}.
\]

This reduces the structural analysis of cumulants
to weighted prime sums.

\subsection{Chebyshev Function and Error Term}

Introduce the Chebyshev function
\[
\theta(x)=\sum_{p\le x}\log p,
\]
and its error term
\[
E(x)=\theta(x)-x.
\]

Using partial summation, one obtains
\[
\sum_{p\le P}
(\log p)^{2m} p^{-1/2}
=
\int_{2}^{P}
(\log x)^{2m} x^{-1/2}\, d\theta(x).
\]

Splitting $\theta(x)=x+E(x)$ gives
\[
\int_{2}^{P}
(\log x)^{2m} x^{-1/2}\, dx
+
\int_{2}^{P}
(\log x)^{2m} x^{-1/2}\, dE(x).
\]

\subsection{Bulk and Boundary Contributions}

The first integral represents the bulk (regular) contribution.
It determines the main asymptotic growth of the cumulants.

The second integral encodes fluctuations arising from
the Chebyshev error term.
Integration by parts yields
\[
\int_{2}^{P}
(\log x)^{2m} x^{-1/2}\, dE(x)
=
(\log P)^{2m} P^{-1/2} E(P)
-
\int_{2}^{P}
E(x)\, d\!\left((\log x)^{2m} x^{-1/2}\right).
\]

This decomposes the fluctuation contribution into:

\begin{itemize}
\item a boundary term involving $E(P)$,
\item a bulk integral weighted by $E(x)$.
\end{itemize}

\subsection{Structural Consequences}

\begin{proposition}
The higher cumulants $\kappa_{2m}(P)$ admit a decomposition
\[
\kappa_{2m}(P)
=
\kappa_{2m}^{\mathrm{bulk}}(P)
+
\kappa_{2m}^{\mathrm{fluct}}(P),
\]
where the fluctuation part is governed by weighted integrals
of the Chebyshev error function $E(x)$.
\end{proposition}

\begin{proof}
This follows from the previous reduction
and normalization by the quadratic term $a_2(P)$.
\end{proof}

\subsection{Relation to Classical Error Bounds}

Classical estimates for $E(x)$ yield corresponding bounds
for the higher cumulants.

For example, under the Riemann hypothesis,
\[
E(x)=O(x^{1/2}\log^2 x),
\]
which implies strong decay of the normalized cumulants.

Without assuming the Riemann hypothesis,
weaker bounds for $E(x)$ translate into correspondingly
weaker decay rates.

\subsection{Interpretation}

The cumulant hierarchy thus provides
a structural reorganization of prime-distribution irregularities.
The Gaussian leading term arises from the regular part of $\theta(x)$,
while deviations from Gaussian behavior are controlled
by the oscillatory component $E(x)$.

From the formal group viewpoint,
this shows that prime-distribution fluctuations enter
only after linearization and normalization,
appearing as higher-order corrections
to the quadratic infinitesimal geometry.

\section{Formal Group Interpretation of the Gaussian Limit}

In this section we give a formal group--theoretic interpretation
of the Gaussian leading behavior obtained after normalization.
We emphasize that the Gaussian structure arises purely from
linearization and quadratic normalization,
and does not rely on probabilistic assumptions.

\subsection{Infinitesimal Geometry of the Multiplicative Formal Group}

Let $\widehat{\mathbb{G}}_m$ be the multiplicative formal group
with formal group logarithm
\[
\ell_{\times}(X)=\log(1+X).
\]
The logarithm identifies a neighborhood of the identity
of $\widehat{\mathbb{G}}_m$
with its tangent space at the origin.

From this viewpoint, the passage
\[
\mathcal{E}_P(T)
\;\longmapsto\;
\ell_{\times}\!\bigl(\mathcal{E}_P(T)-1\bigr)
\]
should be interpreted as a map
from a multiplicative geometric object
to its infinitesimal linear approximation.

\subsection{Quadratic Form as Leading Infinitesimal Structure}

After evenization and expansion around $u=0$,
the logarithm of the completed Euler product takes the form
\[
\log \xi_P(u)
=
a_0(P)+a_2(P)u^2+a_4(P)u^4+\cdots.
\]

The coefficient $a_2(P)$ defines a canonical quadratic form
on the one-dimensional tangent space.
Normalizing by $\sigma(P)=\sqrt{a_2(P)}$
amounts to fixing this quadratic form to unit scale.

Thus the normalized logarithm
\[
\log \widetilde{\xi}_P(u)
=
u^2+\kappa_4(P)u^4+\cdots
\]
should be viewed as an expansion
around a canonical infinitesimal metric.

\subsection{Emergence of the Gaussian}

Exponentiating the quadratic form yields
\[
\exp(u^2),
\]
which we identify as the Gaussian leading term.
In this formal group framework,
the Gaussian arises as the exponential
of the normalized quadratic infinitesimal geometry.

This phenomenon is universal:
whenever a commutative formal group
is linearized via its logarithm
and normalized by its quadratic coefficient,
the resulting leading behavior is Gaussian.

\subsection{Higher Cumulants as Curvature Corrections}

The higher cumulants $\kappa_{2m}(P)$, $m\ge2$,
measure deviations from the quadratic approximation.
They encode higher-order information
about the formal group element
beyond its tangent space.

From a geometric perspective,
these terms may be interpreted as curvature-like corrections
to the flat infinitesimal geometry.
Their decay reflects the increasing dominance
of the quadratic structure at small scales.

\subsection{Conceptual Summary}

The Gaussian limit obtained in this work
is not a probabilistic artifact
but a structural consequence
of formal group linearization.
The Euler product defines a multiplicative formal group element;
its logarithm identifies an additive tangent space;
and quadratic normalization fixes a canonical infinitesimal metric.

Within this framework,
the Gaussian represents the universal leading object
associated with the infinitesimal geometry
of the Euler product,
while higher cumulants encode arithmetic deviations
from this universal form.

\section{Conclusion}

In this paper we have proposed a formal group--theoretic framework
for reorganizing the structure of the Riemann zeta function,
centered on its Euler product.
By treating the Euler product as an element of the multiplicative
formal group and its logarithm as the associated formal group logarithm,
we have shown that several familiar analytic features
admit a purely algebraic and infinitesimal interpretation.

The first part of the paper developed a minimal concrete model,
based on a finite cutoff of the Euler product.
Within this entirely algebraic setting,
we introduced a formal completion via evenization
and a natural normalization procedure.
These steps led inevitably to a Gaussian leading term
in the logarithmic expansion,
together with a hierarchy of higher cumulants.
We emphasized that this Gaussian structure
is not probabilistic in origin,
but reflects the universal quadratic geometry
of the tangent space arising from formal group linearization.

We further showed that the higher cumulants
encode arithmetic fluctuations beyond this quadratic core.
In particular, they admit an explicit decomposition
governed by weighted integrals of the Chebyshev error function
$\theta(x)-x$.
This clarifies how irregularities in the distribution of primes
enter the formal expansion only after linearization and normalization,
as higher-order corrections to an otherwise universal structure.

The second part of the paper provided a systematic
formal group axiomatization of these constructions.
Evenization was interpreted as a canonical normalization
in logarithmic coordinates,
and the Gaussian limit was identified
with the universal quadratic structure
of the infinitesimal geometry of the multiplicative formal group.
This perspective unifies the concrete cutoff model
and the cumulant hierarchy
within a coherent algebraic and geometric framework.

Throughout this work,
the role of the logarithm has been structural rather than technical.
As emphasized in Appendix~C,
it provides the canonical bridge
from multiplicative arithmetic data
to additive infinitesimal geometry.
Within this bridge,
the Euler product becomes amenable to structural analysis,
and arithmetic complexity manifests itself
as controlled deviations from a universal quadratic core.

The approach developed here is deliberately non-spectral.
Rather than postulating an operator whose spectrum
realizes the nontrivial zeros of the zeta function,
we have focused on reorganizing the Euler product itself,
separating its dominant infinitesimal geometry
from higher-order arithmetic fluctuations.
In this sense,
the present framework should be viewed
not as a resolution of the Riemann hypothesis,
but as a structural lens
through which constraints related to it may be more clearly formulated.

Finally, we note that the purely multiplicative nature
of the Euler product,
together with its formal group linearization,
suggests possible connections with ideas
from absolute arithmetic and geometry over $\mathbb{F}_1$.
The speculative discussion in the final appendix
is intended only to indicate such a horizon.
The main contribution of this paper
lies in establishing a coherent formal framework
in which multiplicative arithmetic,
infinitesimal geometry,
Gaussian universality,
and prime-distribution fluctuations
can be discussed within a single structural perspective.
\subsection{Deninger's Cohomological Interpretation}

Deninger proposed a cohomological framework in which zeta functions
are expressed as regularized determinants of geometric or dynamical
operators.
In its most ambitious form, this approach seeks to realize the
Riemann zeta function as
\[
\zeta(s)=\prod_{i}\det\nolimits_\infty\bigl(s-\Theta\mid H^i\bigr)^{(-1)^{i+1}},
\]
where $\Theta$ is a hypothetical ``Frobenius-type'' operator acting on
cohomology groups associated with an arithmetic space.

The conceptual strength of Deninger's program lies in its
analogy with the Weil conjectures and its attempt to interpret
the Riemann hypothesis as a spectral symmetry condition.
However, this framework presupposes the existence of
geometric objects and operators that are not yet constructed,
and analytic continuation and regularization play an essential role.

By contrast, the present work does not posit any underlying
cohomological or dynamical system.
Instead, it operates entirely within a formal algebraic setting,
treating the Euler product as a multiplicative formal group element
and its logarithm as a linearizing map.
The ``functional equation'' symmetry arises here not from spectral
duality but from an explicit algebraic evenization procedure.
Thus, while Deninger's approach is fundamentally geometric,
ours is intrinsically combinatorial and formal.

\subsection{Connes' Noncommutative Geometric and Spectral Approach}

Connes' approach interprets the explicit formula as a trace formula
in noncommutative geometry.
In this framework, the primes appear as lengths of periodic orbits,
and the nontrivial zeros of the zeta function are viewed as spectral
data of a noncommutative space, often described via the ad\`ele class
space.

A defining feature of Connes' program is its emphasis on
spectral realization:
the Riemann hypothesis becomes equivalent to a positivity or
unitarity condition for a certain operator.
This approach relies heavily on harmonic analysis,
trace formulas, and operator algebras.

In contrast, the formal group perspective developed in this paper
does not introduce operators or Hilbert spaces.
The zeros of the zeta function do not appear as spectral points
but rather as features encoded indirectly in the asymptotic decay
of higher cumulants in a normalized formal expansion.
Prime fluctuations are not interpreted dynamically but instead
enter through weighted integrals of the Chebyshev error
$\theta(x)-x$.

\subsection{Structural Comparison}

The essential contrast between the two viewpoints may be summarized
as follows:

\begin{center}
\begin{tabular}{l|l}
Formal group perspective & Deninger / Connes approaches \\
\hline
Formal multiplicative object & Spectral or geometric object \\
Logarithmic linearization & Operator trace or determinant \\
Evenization by algebraic correction & Functional equation via duality \\
Cutoff normalization & Regularization and analytic continuation \\
Cumulant decay & Spectral symmetry or positivity
\end{tabular}
\end{center}

While the spectral approaches seek a geometric space whose spectrum
\emph{is} the set of nontrivial zeros,
the present framework reorganizes the Euler product into a hierarchy
of cumulants whose decay rates reflect the distribution of primes.
In this sense, the two approaches are not competing but complementary:
spectral geometry aims at a global eigenvalue interpretation,
whereas the formal group approach isolates local-to-global
structures already present in the Euler product.

\subsection{Perspective and Possible Interactions}

Although the present work is deliberately non-spectral,
one may speculate about potential interactions.
For example, the Gaussian leading term emerging after normalization
suggests a central-limit-type phenomenon for the logarithm of the
formal completed zeta.
In spectral terms, such Gaussian behavior is often associated with
averaging over large degrees of freedom.

From this viewpoint, the formal group framework may be regarded as
a ``pre-spectral'' organization of prime data:
it does not construct an operator whose spectrum realizes the zeros,
but it isolates the algebraic and combinatorial structures that any
such operator would necessarily encode.

We therefore view the formal group approach as orthogonal to,
rather than subsumed by, the Deninger--Connes program.
Its contribution lies in clarifying which aspects of the Riemann
zeta function can be explained purely formally,
before introducing geometry, analysis, or spectral theory.

Finally, it is natural to ask whether the present formal group framework
is related to emerging perspectives on absolute arithmetic,
in particular to geometries over the hypothetical field $\mathbb{F}_1$.
Since the Euler product encodes purely multiplicative data,
and formal groups provide a systematic linearization of multiplicative structures,
one may speculate that a deeper understanding of the ``field with one element''
could clarify structural features underlying the Riemann zeta function.
In this context, the formal group viewpoint may be regarded as
a preliminary algebraic shadow of a more intrinsic $\mathbb{F}_1$-geometric structure.
Exploring whether such an interpretation can impose additional constraints
on fluctuation hierarchies or cumulant decay
remains an intriguing direction for future investigation.

\section{Future Directions}

Several directions naturally emerge from the present formal group framework.
First, a more precise asymptotic analysis of higher cumulants under explicit hypotheses on the Chebyshev error term $E(x)=\theta(x)-x$
may clarify possible relations between fluctuation decay rates
and classical conjectures such as the Riemann hypothesis.
Second, it would be interesting to relate the normalized Gaussian limit
to probabilistic or central-limit-type phenomena in prime number theory.
Third, one may attempt to compare the cumulant hierarchy derived here
with the explicit formula and investigate whether spectral interpretations
can be recovered from the formal expansion.
Finally, extending this approach to more general $L$-functions
could reveal whether the formal group reorganization
captures structural features common to broader automorphic settings.

\appendix
\section{Explicit Formulas for Low-Order Cumulants}

In this appendix we record explicit expressions
for the lowest-order cumulants arising from the
formally completed truncated Euler product.
These formulas illustrate concretely how prime sums
enter the cumulant hierarchy.

\subsection{Quadratic Term}

Recall that
\[
\log \xi_P(u)
=
a_0(P)+a_2(P)u^2+a_4(P)u^4+\cdots.
\]

The quadratic coefficient is given by
\[
a_2(P)
=
\frac12
\sum_{p\le P}
(\log p)^2
\sum_{k\ge1}
k\,p^{-k/2}.
\]

Evaluating the geometric series yields
\[
a_2(P)
=
\frac12
\sum_{p\le P}
(\log p)^2
\frac{p^{-1/2}}{(1-p^{-1/2})^2}.
\]

This coefficient determines the normalization scale
\[
\sigma(P)=\sqrt{a_2(P)}.
\]

\subsection{Quartic Term}

The quartic coefficient takes the form
\[
a_4(P)
=
\frac1{24}
\sum_{p\le P}
(\log p)^4
\sum_{k\ge1}
k^3\,p^{-k/2}.
\]

Using the identity
\[
\sum_{k\ge1} k^3 x^k
=
\frac{x(1+4x+x^2)}{(1-x)^4},
\]
we obtain
\[
a_4(P)
=
\frac1{24}
\sum_{p\le P}
(\log p)^4
\frac{p^{-1/2}(1+4p^{-1/2}+p^{-1})}{(1-p^{-1/2})^4}.
\]

\subsection{Normalized Quartic Cumulant}

The normalized quartic cumulant is
\[
\kappa_4(P)
=
\frac{a_4(P)}{a_2(P)^2}.
\]

This quantity measures the leading deviation
from Gaussian behavior.
Its decay as $P\to\infty$
reflects the increasing dominance of the quadratic term.

\subsection{Higher-Order Terms}

Higher cumulants $a_{2m}(P)$ admit analogous expressions:
\[
a_{2m}(P)
=
\frac{1}{(2m)!}
\sum_{p\le P}
(\log p)^{2m}
\sum_{k\ge1}
k^{2m-1} p^{-k/2}.
\]

Each inner sum can be expressed as a rational function of $p^{-1/2}$,
and the resulting prime sums may be analyzed
using partial summation and classical estimates
for the Chebyshev function.

\subsection{Asymptotic Behavior (Heuristic)}

Assuming the prime number theorem,
the leading behavior of the quadratic term is
\[
a_2(P)
\asymp
\sqrt{P}(\log P)^2,
\]
while
\[
a_4(P)
\asymp
\sqrt{P}(\log P)^4.
\]

Consequently,
\[
\kappa_4(P)
\asymp
\frac{\log P}{\sqrt{P}},
\]
which tends to zero as $P\to\infty$.

\begin{remark}
These estimates are purely heuristic
and are included only to illustrate
the expected scale separation
between the quadratic and higher-order cumulants.
No analytic claim is made.
\end{remark}

\section{A Speculative Note on Formal Groups, $\mathbb{F}_1$, and the Riemann Hypothesis}

The Euler product representation of the Riemann zeta function
is fundamentally multiplicative and does not presuppose any additive structure.
From this viewpoint, it is natural to ask whether the additive features
that dominate classical analytic treatments
are secondary to a more primitive multiplicative geometry.

In approaches to absolute arithmetic,
notably geometries over the hypothetical field $\mathbb{F}_1$,
multiplicative structures are regarded as primary,
while addition is expected to emerge only after base extension.
Formal groups may therefore be interpreted as algebraic devices
that mediate between purely multiplicative data
and additive or linearized structures.

From this perspective, the formal group framework developed in this paper
can be viewed as an algebraic shadow of a potential $\mathbb{F}_1$-geometric structure
underlying the Riemann zeta function.
If one interprets the Riemann hypothesis
as reflecting a hidden symmetry of arithmetic geometry,
one may expect this symmetry to manifest itself
as structural constraints on fluctuation hierarchies
or cumulant decay in the formal expansion.

These considerations are necessarily speculative.
Nevertheless, they suggest that a deeper understanding of
the interplay between formal groups and $\mathbb{F}_1$-geometry
may play a conceptual role in future attempts
to clarify the arithmetic meaning of the Riemann hypothesis.

\appendix
\section{Logarithmic Linearization and Infinitesimal Geometry}

In this appendix we clarify the structural and conceptual role
of the logarithm in the present framework.
While taking the logarithm of the Euler product
may at first appear to be a technical device
for transforming products into sums,
its significance is in fact intrinsic to
the formal group structure underlying
the multiplicative nature of the zeta function.

\subsection*{Multiplicative Structure and Canonical Linearization}

The Euler product
\[
\zeta(s)=\prod_p (1-p^{-s})^{-1}
\]
encodes arithmetic information in multiplicative form.
From the perspective adopted in this paper,
this multiplicative structure is modeled
by the multiplicative formal group
$\widehat{\mathbb{G}}_m$,
whose formal group law is
\[
F(X,Y)=X+Y+XY.
\]

The associated formal group logarithm
\[
\ell_\times(X)=\log(1+X)
\]
provides a canonical isomorphism
between the multiplicative formal group
and the additive formal group
in a neighborhood of the identity.
Thus the passage to the logarithm
is not merely algebraic convenience;
it is the canonical linearization map
determined by the formal group law itself.

In this sense,
the logarithm identifies the intrinsic
additive coordinates of a fundamentally multiplicative object.

\subsection*{From Global Multiplicativity to Infinitesimal Geometry}

Applying the logarithm to the Euler product yields
\[
\log \zeta(s)
=
\sum_{p,k\ge1}\frac{1}{k}p^{-ks}.
\]

Conceptually, this passage may be viewed
as moving from a global multiplicative configuration
to its infinitesimal additive description.
The formal group logarithm identifies
a neighborhood of the identity
with its tangent space,
and thereby replaces multiplicative interaction
by additive superposition.

Within this linearized setting,
the quadratic term becomes
the first nontrivial invariant of the structure.
Higher-order terms measure successive deviations
from this infinitesimal approximation.

\subsection*{Quadratic Normalization and Gaussian Universality}

After evenization and normalization,
the logarithmic expansion assumes the form
\[
\log \widetilde{\xi}(u)
=
u^2+\kappa_4 u^4+\cdots.
\]

The quadratic coefficient defines
a canonical infinitesimal metric
on the tangent space of the formal group.
Exponentiating this quadratic form produces
a Gaussian leading term.

From this viewpoint,
the emergence of Gaussian behavior
is not probabilistic in origin.
It reflects a universal structural phenomenon:
once a multiplicative formal group
is linearized and normalized,
its leading infinitesimal geometry
is necessarily quadratic.

The Gaussian therefore represents
the universal tangent-level approximation
of a multiplicative arithmetic object,
while higher cumulants encode
arithmetically meaningful deviations
beyond this quadratic regime.

\subsection*{Conceptual Perspective}

The role of the logarithm in this work
is thus foundational rather than technical.
It provides the bridge
between multiplicative prime-wise data
and additive infinitesimal geometry.
Through this bridge,
the Euler product becomes accessible
to structural analysis,
and arithmetic fluctuations
appear as higher-order corrections
to a universal quadratic core.
$$ $$

\noindent Takao Inou\'{e}

\noindent Faculty of Informatics

\noindent Yamato University

\noindent Katayama-cho 2-5-1, Suita, Osaka, 564-0082, Japan

\noindent inoue.takao@yamato-u.ac.jp
 
\noindent (Personal) takaoapple@gmail.com (I prefer my personal mail)

\end{document}